\documentclass[11pt, reqno]{amsart}
\usepackage{amsmath}
\usepackage{amssymb}
\usepackage{fullpage}
\usepackage{amsthm}
\usepackage{graphicx}
\usepackage{placeins}
\usepackage{caption}
\usepackage{extpfeil}
\usepackage{enumerate}
\usepackage{caption}
\usepackage{subcaption, esint}
\usepackage{hyperref}
\usepackage{xcolor}
\usepackage{pbox}

\definecolor{blue}{rgb}{0,0,0.9}
\definecolor{purple}{rgb}{0.6,0,0.9}

\newcommand{\be}{\begin{equation}}
\newcommand{\ee}{\end{equation}}
\newcommand{\bee}{\begin{equation*}}
\newcommand{\eee}{\end{equation*}}

\newcommand{\supp}{\text{\textup{supp}}\,}

\newtheorem{theorem}{Theorem}[section]
\newtheorem{proposition}[theorem]{Proposition}

\newtheorem{lemma}[theorem]{Lemma}

\theoremstyle{definition}
\newtheorem{remark}[theorem]{Remark}
\theoremstyle{definition}
\newtheorem*{remark*}{Remark}

\numberwithin{equation}{section}

\begin{document}

\title{Sharp stability for the interaction energy}
\author{Xukai Yan \,and\, Yao Yao
}

\begin{abstract}
This paper is devoted to stability estimates for the interaction energy with strictly radially decreasing interaction potentials, such as the Coulomb and Riesz potentials. For a general density function, we first prove a stability estimate in terms of the $L^1$ asymmetry of the density, extending some previous results by Burchard--Chambers \cite{BC, BC2}, Frank--Lieb \cite{FL} and Fusco--Pratelli \cite{FP} for characteristic functions. We also obtain a stability estimate in terms of the 2-Wasserstein distance between the density and its radial decreasing rearrangement. Finally, we consider the special case of Newtonian potential, and address a conjecture by Guo on the stability for the Coulomb energy.
\end{abstract}
\maketitle

\section{Introduction}

For a density $\rho \in L^1_+(\mathbb{R}^n) \cap L^\infty(\mathbb{R}^n)$, let $\mathcal{E}_W[\rho]$ be the interaction energy of $\rho$ with interaction potential $W$, given by
\[
\mathcal{E}_W[\rho] :=  \int_{\mathbb{R}^n} \int_{\mathbb{R}^n}\rho(x) \rho(y) W(x-y)\,dx dy.
\]  
Throughout this paper, we focus on potentials $W\in C^1(\mathbb{R}^n\setminus\{0\}) \cap L^1_{loc}(\mathbb{R}^n)$ that are radially symmetric, and strictly decreasing in the radial variable. As an important special case, when $n=3$ and $W=\mathcal{N}=\frac{1}{4\pi|x|}$ is the Newtonian potential in $\mathbb{R}^3$, $\mathcal{E}_\mathcal{N}[\rho]$ represents the Coulomb energy of a charge with density $\rho$.

Let $\rho^*$ be the radially symmetric decreasing rearrangement of $\rho$; see \cite[Section 3.3]{LiebLoss} for a definition.
The celebrated Riesz's rearrangement inequality \cite[Section 3.7]{LiebLoss} gives
\begin{equation}\label{riesz}
\mathcal{E}_W[\rho^*] \geq \mathcal{E}_W[\rho],
\end{equation}
with equality achieved if and only if $\rho$ is equal to $\rho^*$ almost everywhere after a translation \cite{Lieb}.  Heuristically speaking, \eqref{riesz} describes that symmetrization reduces the typical distance between the charges, thus increases the interaction energy.

The goal of this paper is to improve \eqref{riesz} to a quantitative version: if its two sides almost agree, how close must $\rho$ be to a translation of $\rho^*$? In other words, we want to obtain a stability estimate of the form
\begin{equation}\label{ineq_all}
\mathcal{E}_W[\rho^*] - \mathcal{E}_W[\rho] \geq d(\rho,\rho^*) \geq 0,
\end{equation} where $d(\rho,\rho^*)$  measures the ``asymmetry'' of $\rho$, and should equal to 0 if and only if $\rho$ agrees with $\rho^*$ almost everywhere after a translation. Throughout this paper, we denote by $T_a$ the operator that translates a function by a vector $a\in\mathbb{R}^n$, that is, $T_a \rho := \rho(\cdot+a)$. 

Let us review some previous results on stability estimates for \eqref{riesz}. A natural way to measure the distance between $\rho$ and $\rho^*$ is to consider their minimum $L^1$ distance among all translations $T_a$, and then normalize it, i.e.
\begin{equation}\label{def_delta_rho}
\delta(\rho) := \inf_{a\in\mathbb{R}^n} \frac{\| T_a \rho - \rho^*\|_{L^1(\mathbb{R}^n)}}{2\|\rho\|_{L^1(\mathbb{R}^n)}},
\end{equation}
where the factor 2 in the denominator ensures $\delta(\rho)\in[0,1)$.
 When $\rho=1_D$ is a characteristic function, and $W=\mathcal{N}$ is the Newtonian potential in $\mathbb{R}^3$, Burchard and Chambers \cite{BC} obtained a sharp stability estimate for $\mathbb{R}^3$ that
\begin{equation}\label{bc}
\mathcal{E}_{\mathcal{N}}[1_D^*]-\mathcal{E}_{\mathcal{N}}[1_D]\geq c |D|^{\frac{5}{3}}\,\delta(1_D)^2
\end{equation}
for some $c>0$,
where the estimate is sharp in the sense that the exponent 2 on the right hand side cannot be replaced by any smaller ones. They also used a different approach to obtain a stability estimate for Newtonian potentials in $\mathbb{R}^n$ for $n>3$, however the exponent $2$ is replaced by a non-sharp exponent $n+2$.

Recently, Frank and Lieb \cite{FL}, Fusco and Pratelli \cite{FP}, and Burchard and Chambers \cite{BC2} independently  obtained sharp stability estimates for the interaction energy in all dimensions for power-law potentials $W_k$, given by
\begin{equation}\label{def_riesz}
W_k(x) :=\begin{cases} -\dfrac{|x|^{k}}{k}& \text{ for }k\neq 0,\\
-\log|x| &\text{ for } k=0.
\end{cases}
\end{equation}
For $n\geq 2$, when $\rho=1_D$ is a characteristic function\footnote{The original statement in \cite{FP} focused on $\rho=1_D$ with $|D|=1$, but after a simple scaling argument one gets \eqref{fp} for any $|D|>0$.}, Fusco and Pratelli \cite{FP} proved that there exists some $c(n,k)>0$ such that
\begin{equation}\label{fp}
\mathcal{E}_{W_k}[1_D^*]-\mathcal{E}_{W_k}[1_D]\geq c(n,k) |D|^{2+\frac{k}{n}}\delta(1_D)^2\quad\text{ for } k\in(-n+1,0),
\end{equation}
using a delicate combination of geometrical and mass transportation arguments. Very recently \eqref{fp} is also obtained by Burchard and Chambers \cite{BC2} using a Fuglede-type estimate \cite{Fuglede} in combination with global rearrangements.
The proofs of  \eqref{bc} and \eqref{fp}  both strongly rely on the assumption that $\rho=1_D$ is a characteristic function, and could not be easily extended to general densities. 

For $n\geq 1$ and a general density function $\rho$ satisfying $0\leq \rho\leq 1$, Frank and Lieb \cite[Theorem~4--5]{FL} obtained the following inequality comparing $\mathcal{E}_{W_k}[\rho]$ with $\mathcal{E}_{W_k}[1_{E^*}]$, where $E^*$ is a ball centered at the origin with $|E^*|=\int_{\mathbb{R}^n} \rho dx$:
\begin{equation}\label{fl_ineq}
\mathcal{E}_{W_k}[1_{E^*}]-\mathcal{E}_{W_k}[\rho]\geq c(n,k) \|\rho\|_{L^1}^{2+\frac{k}{n}}\,\tilde\delta(\rho,1_{E^*})^2\quad\text{ for } k\in(-n,\infty),
\end{equation}
where $\tilde\delta(\rho,1_{E^*})$ is defined the same as $\delta(\rho)$ except that $\rho^*$ is replaced by $1_{E^*}$.
Their proof is built on a deep result by Christ \cite{christ, FL2} on stability estimates for $\mathcal{E}_{1_B}[1_{E^*}] - \mathcal{E}_{1_B}[\rho]$, where $B$ is a ball centered at the origin. Note that in the special case when $\rho=1_D$ is a characteristic function, \eqref{fl_ineq} directly becomes \eqref{fp} (for a broader range of $k$) since $1_{E^*}=1_{D}^*$. However, for general densities with $0\leq \rho\leq 1$,  \eqref{fl_ineq} does not imply a stability estimate for $\mathcal{E}_{W_k}[\rho^*]-\mathcal{E}_{W_k}[\rho]$, since $\mathcal{E}_{W_k}[\rho^*] \leq \mathcal{E}_{W_k}[1_{E^*}]$.

To our best knowledge, we are unaware of any stability estimates of $\mathcal{E}_{W}[\rho^*]-\mathcal{E}_{W}[\rho]$ for densities that are not characteristic functions. The following results deal with stability estimates for other nonlocal functionals (also restricted to either sets or characteristic functions): Carlen and Maggi \cite[Theorem 1.5]{CM} obtained stability estimates for Riesz's rearrangement inequality for two different characteristic functions $1_E$ and $1_F$, of the form $\int_{E^*} \int_{F^*} W(x-y) dxdy -  \int_{E} \int_{F} W(x-y) dxdy \gtrsim \delta(1_E)^{8(n+2)}$. For stability properties of balls with respect to nonlocal energies involving the fractional perimeter, see \cite{DNRV, FFMMM, GNR}.

One goal of our paper is to extend \eqref{fp} to a general density $\rho \in L_+^1(\mathbb{R}^n) \cap L^\infty(\mathbb{R}^n)$; see Theorem~\ref{theorem1} below. The theorem holds for all power-law potentials $W_k$ with $k\in(-n,2]$. Our proof strategy is completely different from those in \cite{BC, BC2, FL, FP}; in fact our approach is quite elementary.

Another natural question is whether one can use another distance different from the $L^1$ norm to measure $d(\rho,\rho^*)$ in \eqref{ineq_all}. Our second main result is a stability estimate of the interaction energy in terms of the 2-Wasserstein distance between $\rho^*$ and a translation of $\rho$; see Theorem~\ref{theorem2} below. The 2-Wasserstein distance naturally arises in many studies of the interaction energy, see \cite{AGS, CMV, CMV2, McCann, Villani}.

Other than the $L^1$ distance and 2-Wasserstein distance, we aim to obtain a third stability estimate with $d(\rho,\rho^*)$ given by the interaction energy itself for the special case of Newtonian potential. Note that
when $\mathcal{N}$ is the Newtonian potential in $\mathbb{R}^n$, $\mathcal{E}_\mathcal{N}$ is positive definite in the sense that for any $ f\in L^1(\mathbb{R}^n) \cap L^\infty(\mathbb{R}^n)$ (which can be sign-changing), one has
\begin{equation}\label{eq_newtonian}
\mathcal{E}_\mathcal{N}[f] = \int_{\mathbb{R}^n} f (-\Delta)^{-1}f dx 
= \|f\|_{\dot H^{-1}}^2 \geq 0.
\end{equation}
Motivated by the positive definiteness of $\mathcal{E}_\mathcal{N}$, Yan Guo conjectured (see \cite[Eq.(3)]{BC}) that whether the following inequality holds for all $\rho \in L^1_+(\mathbb{R}^n)\cap L^\infty(\mathbb{R}^n)$:
\begin{equation}\label{conj_guo}
\mathcal{E}_\mathcal{N}[\rho^*] - \mathcal{E}_\mathcal{N}[\rho] \overset{?}{\geq} c(n) \inf_a \mathcal{E}_\mathcal{N}[T_a \rho - \rho^*].
\end{equation}
Note that no normalization is required as we scale or dilate $\rho$, because both sides scale in the same way. To our best knowledge, the answer to the conjecture remained open.  (See \cite[Theorem 1]{BG} for some related results on a sequence of functions without a quantitative estimate.) Another goal of our paper is to prove that \eqref{conj_guo} would be correct if we replace $c(n)$ by $c(n, \|\rho\|_{L^1}, \|\rho\|_{L^\infty}, R)$, where $R$ is the smallest number such that $\text{supp}\,\rho\subset B(0,R)$. We will also construct counterexamples to show that the dependence on $\|\rho\|_{L^1}, \|\rho\|_{L^\infty}$ and $R$ is indeed necessary for $n\geq 3$.

\subsection{Our results}

Throughout this paper, we assume that $W$ satisfies the following assumptions:

\noindent \textbf{(W1)} $W\in C^1(\mathbb{R}^n\setminus\{0\}) \cap L^1_{loc}(\mathbb{R}^n)$ is radially symmetric with $W(x)=w(|x|)$ for some $w:\mathbb{R}^+ \to \mathbb{R}$.

\noindent \textbf{(W2)} $w'(r)<0$ for $r>0$. 

\noindent \textbf{(W3)}  $W$ is not too flat near the origin: Namely, there exists some $c>0$, such that $w'(r)\leq -cr$ for $r\in(0,1)$.

In particular, note that for $k\in (-n,2]$, the power-law potentials $W_k$ given by \eqref{def_riesz} satisfy all the above assumptions. However, $W_k$ with $k>2$ violates \textbf{(W3)}, due to being too flat near the origin.

Below we state our results. The first main result improves \eqref{fp} to general densities $\rho\in L_+^1(\mathbb{R}^n)\cap L^\infty(\mathbb{R}^n)$, using a completely different (and in fact more elementary) proof from those in \cite{BC, BC2, FL, FP}.
\begin{theorem}\label{theorem1}
Assume that $W$ satisfies \textbf{\textup{(W1)}}--\textbf{\textup{(W3)}}. Let $\rho\in L_+^1(\mathbb{R}^n)\cap L^\infty(\mathbb{R}^n)$, with $\supp \rho^* \subset B(0,R_*)$ for some finite $R_*>0$. Then we have the following stability estimate for $n\geq 2$:
\begin{equation}\label{thm1_eq1}
 \mathcal{E}_W[\rho^*] - \mathcal{E}_W[\rho] \geq c(n,W, R_*) \|\rho\|_1^{2+\frac{2}{n}} \|\rho\|_\infty^{-\frac{2}{n}}\delta(\rho)^2.
\end{equation}
In particular, for power-law potentials $W_k$ with $k\in(-n,2]$, the constant is 
$
 c(n,W_k, R_*) = c(n) R_*^{k-2}.$

For $n=1$, \eqref{thm1_eq1} also holds if $\rho = c 1_D$ is a multiple of a characteristic function. If $n=1$ and $\rho$ is a general density in $L_+^1(\mathbb{R})\cap L^\infty(\mathbb{R})$, then \eqref{thm1_eq1} holds with $\delta(\rho)^2$ replaced by $\delta(\rho)^3$.

\end{theorem}

\noindent \textbf{Remarks.} 1. We do not need $\rho$ to be compactly supported, but it is indeed necessary to assume $\rho^*$ is supported in $B(0,R_*)$. For $W=W_k$ with $k\in(-n,2]$, the constant $c(n,W_k, R_*) = c(n) R_*^{k-2}$ has the sharp power on $R_*$. To see this, for any characteristic function $\rho = 1_D$, $c(n,W_k, R_*)$ becomes $c(n) |D|^{\frac{k-2}{n}}$ due to $\omega_n R_*^n = |D|$, thus
 \eqref{thm1_eq1} exactly becomes the sharp estimate \eqref{fp}.

2. For $n=1$ and a general density $\rho \in L_+^1(\mathbb{R})\cap L^\infty(\mathbb{R})$, the power of $\delta(\rho)^3$ in \eqref{thm1_eq1} is indeed sharp. We will construct examples in Remark~\ref{rmk_1d_2} to demonstrate the sharpness of power 3.

\medskip
Our second main result is also a stability result for the interaction energy. The novelty is that instead of using the $L^1$ norm to measure the ``asymmetry'' of $\rho$, we now use the 2-Wasserstein distance $W_2$, which is the natural metric to use in many studies of the interaction energy, see \cite{McCann, CMV, CMV2},  \cite[Section 5.2.5]{Villani}, and \cite[Section 10.4.5]{AGS}. Since $W_2$ is defined among probability measures, we assume $\rho$ to be a probability density below. The additional assumption $\rho\in L^\infty(\mathbb{R}^n)$ is only to ensure that $ \mathcal{E}_W[\rho^*]$ and $\mathcal{E}_W[\rho]$ are both finite, and we expect that it can be relaxed.

\begin{theorem}\label{theorem2}Assume that $W$ satisfies \textbf{\textup{(W1)}}--\textbf{\textup{(W3)}}. 
Let $\rho\in \mathcal{P}(\mathbb{R}^n) \cap L^\infty(\mathbb{R}^n)$, with $\supp\rho \subset B(0,R)$ for some finite $R>0$. Denote its center of mass by $x_0 := \int x\rho dx$.   Then 
\begin{equation}\label{eq_w2}
 \mathcal{E}_W[\rho^*] - \mathcal{E}_W[\rho] \geq c(W, R)  \,W_2^2(T_{x_0}\rho, \rho^*).
\end{equation}
In particular, if $W=W_k$ with $k\in(-n,2]$, the constant is given by $c(W_k, R) = (2R)^{k-2}$.
\end{theorem}

\noindent \textbf{Remarks.} 1. Note that the power 2 on the right hand side of \eqref{eq_w2} is sharp: for $0<\epsilon\ll 1$, let $\rho = c(1_{B(0,1)} + 1_{2\leq |x|\leq 2+\epsilon})$, where $c$ is the normalizing constant so that $\int_{\mathbb{R}^n} \rho dx = 1$. As $\epsilon\to 0^+$, one can easily check that $W_2(\rho,\rho^*) \stackrel{n}{\sim} \sqrt{\epsilon}$ and $\mathcal{E}_W[\rho^*] - \mathcal{E}_W[\rho] \stackrel{n,W}{\sim} \epsilon$, thus the power cannot be lowered.

2. In Remark~\ref{rmk_w2}, we show that for potentials $W_k$ with $k\in(-n,2]$, the power $k-2$ in $c(W_k,R)=(2R)^{k-2}$ is also sharp. In addition,  it is necessary to allow $c(W,R)$ to depend on $R$; in particular one cannot replace it by the radius of $\supp\rho^*$ as we did in Theorem~\ref{theorem1}.

\medskip
In our third result we focus on the Newtonian potential $\mathcal{N}$ in $\mathbb{R}^n$, and address Guo's conjecture \eqref{conj_guo}. We first give a positive result, showing that the conjecture is true if we allow the constant on the right hand side of \eqref{conj_guo} to depend on the $L^1, L^\infty$ norm and the support radius of $\rho$.

\begin{theorem}\label{theorem3} Let $\mathcal{N}$ be the Newtonian potential in $\mathbb{R}^n$.
Let $\rho\in L_1^+(\mathbb{R}^n)\cap L^\infty(\mathbb{R}^n)$, with  $\supp\rho \subset B(0,R)$. Denote its center of mass by $x_0 := \frac{\int x\rho dx}{\|\rho\|_1}$.  Then there exists some constant $c(n)>0$ only depending on $n$, such that the following holds:
\begin{equation}\label{eq_thm3}
 \mathcal{E}_\mathcal{N}[\rho^*] - \mathcal{E}_\mathcal{N}[\rho] \geq \frac{c(n)\|\rho\|_{1}}{\|\rho\|_{\infty} R^n} \, \mathcal{E}_\mathcal{N}[T_{x_0}\rho - \rho^*].
\end{equation}
\end{theorem}

\begin{remark*}
Due to the relationship  \eqref{eq_newtonian} between the $\mathcal{E}_\mathcal{N}$ and the $\dot H^{-1}$ norm,  Theorem~\ref{theorem3} immediately leads to the following stability estimate of the $\dot H^{-1}$ norm:
\[
\|\rho^*\|_{\dot H^{-1}}^2 - \|\rho\|_{\dot H^{-1}}^2 \geq \frac{c(n) \|\rho\|_{1}}{\|\rho\|_{\infty} R^n} \, \|T_{x_0}\rho - \rho^*\|_{\dot H^{-1}}^2 \geq 0.
\]
\end{remark*}

One might wonder whether the dependence on $\|\rho\|_{1}, \|\rho\|_{\infty}$ and $R$ is necessary in Theorem~\ref{theorem3}. In our final result, we show that they are necessary for $n\geq 3$: we construct counterexamples in $n\geq 3$ with $\|\rho\|_1 \stackrel{n}{\sim} 1$, $R \sim 1$ and $\|\rho\|_\infty \gg 1$, showing that \eqref{conj_guo} is false if the constant on the right hand side is not allowed to depend on $\|\rho\|_\infty$. (Thus clearly the constant in \eqref{eq_thm3} also needs to depend on $\|\rho\|_\infty$, due to $\mathcal{E}_\mathcal{N}[T_{x_0}\rho - \rho^*] \geq \inf_a \mathcal{E}_\mathcal{N}[T_{a}\rho - \rho^*]$.)

\begin{theorem}\label{theorem4}
Assume $n\geq 3$. For any $\epsilon\in(0,1)$ that is sufficiently small, there exists some $\rho \in L^1_+(\mathbb{R}^n) \cap L^\infty(\mathbb{R}^n)$ supported in $B(0,7)$ with $\omega_n \leq \|\rho\|_{1} \leq 2\omega_n$ and $\|\rho\|_{\infty} = \epsilon^{-(\frac{n}{2}+1)}$, such that
\begin{equation}\label{eq_counterexample}
0<\mathcal{E}_{\mathcal{N}}[\rho^*] - \mathcal{E}_{\mathcal{N}}[\rho] < C(n) \epsilon^{\frac{n}{2}-1} \inf_a \mathcal{E}_\mathcal{N}[T_a \rho - \rho^*],
\end{equation}
where $C(n)$ only depends on $n$. Here $\omega_n$ is the volume of unit ball in $\mathbb{R}^n$.
\end{theorem}

Once the dependence on $\|\rho\|_{\infty}$ is necessary, one can simply dilate and scale $\rho$ to check that the dependence on $\|\rho\|_{1}$ and $R$ is also necessary  in both the conjecture \eqref{conj_guo} and Theorem~\ref{theorem3}, although it is unclear whether the powers of $\|\rho\|_\infty, \|\rho\|_1$ and $R$  in \eqref{eq_thm3} are sharp.

\subsection{Strategy of proof}
\label{sec_strategy}
In all the stability theorems, we make a simple but useful observation that \emph{if} $\rho$ is known to be supported in $B(0,R)$,\footnote{For Theorem \ref{theorem2}--\ref{theorem3}, $\supp\rho\subset B(0,R)$ is already part of the assumption. For Theorem~\ref{theorem1}, although $\rho$ is not assumed to have compact support (recall that we only assume $\supp\rho^*\subset B(0,R_*)$), we will show that the proof can be reduced to the case where $\supp \rho \subset B(0,R)$ for $R=20R_*$.  
} it suffices to obtain stability estimates for a quadratic interaction potential. Namely, the assumptions on $W$ allow us to decompose it as
\[
W(x) = -c(R,W) |x|^2 + \tilde W(x),
\]
where $c(R,W)>0$ and $\tilde W$ is radially decreasing in $B(0,2R)$. Figure~\ref{fig1} shows an illustration of the decomposition, and the explicit choice of $c$ will be given in Section~\ref{sec_quadratic}. This immediately leads to
\[
\mathcal{E}_W[\rho^*] - \mathcal{E}_W[\rho] = c(R,W)\left(\mathcal{E}_{-|x|^2} [\rho^*] - \mathcal{E}_{-|x|^2}[\rho]\right) + \left( \mathcal{E}_{\tilde W}[\rho^*] - \mathcal{E}_{\tilde W}[\rho]\right), 
\]where both parentheses are nonnegative since $\supp \rho \subset B(0,R)$.

\begin{figure}[h!]
\begin{center}
\includegraphics[scale=1]{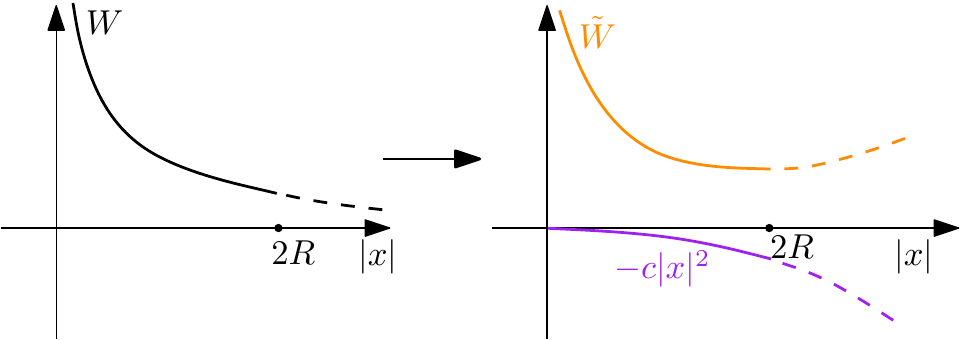}
\caption{Decomposing $W$ into the sum of a (negative) quadratic potential and a potential $\tilde W$ that is radially decreasing in $B(0,2R)$. As illustrated by the dashed orange line, $\tilde W$ is not necessarily radially decreasing outside radius $2R$. Note that $\rho$ and $\rho^*$ cannot ``feel'' the dashed line part of $\tilde W$, since both are supported in $B(0,R)$. \label{fig1}}
\end{center}
\end{figure}

The motivation for us to do this decomposition is that the quadratic interaction potential $|x|^2$ has a very special property  (see \cite{CFT, LM} for example): if $\rho$ has a finite second moment $M_2[\rho] := \int_{\mathbb{R}^n} \rho|x|^2 dx<\infty$, and center of mass $x_0 := \frac{\int \rho x dx}{\|\rho\|_{1}}$ (so its translation $T_{x_0}\rho$ has center of mass at the origin), then
\begin{equation}\label{eq_second_moment}
\mathcal{E}_{|x|^2} [\rho] = 2\|\rho\|_{1} M_2[T_{x_0}\rho].
\end{equation}
Therefore, it suffices to obtain lower bounds of 
$M_2[T_{x_0}\rho] - M_2[\rho^*]$.
Since the second moment $M_2$ is a linear functional, it is much easier to deal with $M_2$ compared to the original interaction energy. We then obtain Theorem~\ref{theorem1} by applying a stability estimate for $M_2$ in terms of $\delta(\rho)$.

In order to prove Theorem~\ref{theorem2}, using the above observation, it remains to relate the second moment difference with the 2-Wasserstein distance. We prove the following sharp stability estimate of the second moment for all $\rho\in\mathcal{P}_2(\mathbb{R}^n)$, which might be of independent interest:
\[
M_2[\rho] - M_2[\rho^*] \geq W_2^2(\rho, \rho^*),
\]
where both the power 2 and the constant 1 on the right hand side are sharp. 

Finally, once we obtain Theorem~\ref{theorem2}, using a remarkable observation by Loeper \cite{Loeper} regarding the connection between the Wasserstein distance and $\dot H^{-1}$ norm, Theorem~\ref{theorem3} follows as a direct consequence.

\subsection{Organization of the paper}
In Section \ref{sec_quadratic} we prove a simple lemma explaining the ``reduction to quadratic interaction potential'' idea in details.
Section~\ref{sec3} and \ref{sec4} are devoted to the proofs of Theorem~\ref{theorem1} and Theorem~\ref{theorem2} respectively. In Section \ref{sec5} we focus on Newtonian potential, and prove Theorem~\ref{theorem3}--\ref{theorem4}.

\subsection{Notations}
Throughout this paper, we denote $\|f\|_p := \|f\|_{L^p(\mathbb{R}^n)}$ for $1\leq p\leq \infty$. 

Let $B(a,r)$ be the ball in $\mathbb{R}^n$ centered at $a$ with radius $r$, and $\omega_n$ is the volume of the unit ball in $\mathbb{R}^n$. For a set $D\subset \mathbb{R}^n$ with finite volume (i.e. $|D|<\infty$), we denote by $D^*$ the symmetric rearrangement of $D$, i.e. $D^*$ is a ball centered at the origin with $|D^*| = |D|$.

We denote by $C(n)$ positive constants only depending on $n$, whose value may change from line to line. Likewise, $C(n,W,R)$ only depends on $C, W$ and $R$.  For two non-negative quantities $f, g$, we write $f\lesssim g$ if $f\leq Cg$ for some universal constant $C>0$. $f\gtrsim g$ is defined likewise. And we write $f\sim g$ if  both $f\lesssim g$ and $g\lesssim f$ hold. If the constant depends on other parameters (such as $n, W,\dots$), we write $f \stackrel{n}{\sim} g$, $f \stackrel{n,W}{\sim} g$.

Let $\mathcal{P}(\mathbb{R}^n)$ denote the set of probability measures in $\mathbb{R}^n$, and $\mathcal{P}_2(\mathbb{R}^n)$ denotes the set of probability measures  in $\mathbb{R}^n$ with finite second moment. (Recall that the second moment of $\rho$ is given by $M_2[\rho] := \int_{\mathbb{R}^n} \rho(x) |x|^2 dx$.) For $\rho_1, \rho_2 \in \mathcal{P}_2(\mathbb{R}^n)$, let $W_2(\rho_1, \rho_2)$ denote their $2$-Wasserstein distance; see \cite[Section 7.1]{AGS} for a definition.

For any $f\in L_+^1(\mathbb{R}^n)$ and a map $T:\mathbb{R}^n \to \mathbb{R}^n$, we define $T\# f$ as the \emph{push-forward} of $f$ under $T$, which satisfies $\int_{A} (T\#f)(x) dx = \int_{T^{-1}(A)} f(x) dx$ for all $A \subset \mathbb{R}^n$.

\subsection*{Acknowledgement}
YY was partially supported by the NSF grants DMS-1715418, DMS-1846745, and Sloan Research Fellowship. The authors would like to thank Yan Guo and Pierre-Emmanuel Jabin for helpful discussions.

\section{Reducing to quadratic interaction potential}\label{sec_quadratic}
The following simple lemma explains the ``reduction to quadratic interaction potential'' idea in Section~\ref{sec_strategy} in details. It shows that if $\supp \rho \subset B(0,R)$, in order to obtain stability estimates for the interaction energy, all we need is a stability estimate of the second moment.

\begin{lemma}\label{lem_reduce_m2}Assume that $W$ satisfies \textbf{\textup{(W1)}}--\textbf{\textup{(W3)}}. For all $\rho\in L^1_+(\mathbb{R}^n)\cap L^\infty(\mathbb{R}^n)$ with $\supp\rho \subset B(0,R)$, there exists $c_{W,R}>0$ only depending on $W$ and $R$, such that
\[
\mathcal{E}_W[\rho^*] - \mathcal{E}_W[\rho] \geq c_{W,R} \|\rho\|_1 (M_2[T_{x_0}\rho] - M_2[\rho^*]),
\]
where $x_0 = \frac{\int x\rho dx}{\|\rho\|_1}$ is the center of mass of $\rho$. In particular, if $W=W_k$ with $k\in(-n,2]$, we have $c_{W_k,R} = (2R)^{k-2}$.

\end{lemma}

\begin{proof}
If $\rho$ is supported in $B(0,R)$,  the distance between any two points in its support does not exceed $2R$. The same is also true for $\rho^*$, since $\supp\rho^*$ is a ball with the same volume as $\supp \rho$. 

 With this in mind, we will split $W$ into the sum of a quadratic potential $-c|x|^2$ (with $c>0$) and another potential $\tilde W$ that is radially decreasing in $B(0,2R)$. 
In order to do this, let
\begin{equation}\label{c_general}
c_{W,R} := \inf_{r\in(0,2R)} \frac{-w'(r)}{r}.
\end{equation}
Note that  $c_{W,R}>0$ by the assumptions \textbf{(W1)}--\textbf{(W3)}. In particular, if $W=W_k$ with $k\in(-n,2]$, one can compute explicitly that 
\begin{equation}\label{eq:cwkr}
c_{W_k,R} = (2R)^{k-2}.
\end{equation}
With the above definition, we have that $w'(r) + c_{W,R}\,r\leq 0$ for $r\in(0,2R)$, thus  
\begin{equation}\label{def_tildeW}\tilde W(x) := W(x)+\frac{c_{W,R}}{2}|x|^2
\end{equation} is radially decreasing in $B(0,2R)$. 

We rewrite $W$ as $W = \tilde W - \frac{c_{W,R}}{2}|x|^2$, and use the linearity of $\mathcal{E}_W$ with respect to $W$ to obtain
\begin{equation}\label{temp_e}
\begin{split}
\mathcal{E}_W[\rho^*] - \mathcal{E}_W[\rho] &= (\mathcal{E}_{\tilde W}[\rho^*] - \mathcal{E}_{\tilde W}[\rho]) - \frac{c_{W,R}}{2} \Big(\mathcal{E}_{|x|^2}[\rho^*] - \mathcal{E}_{|x|^2}[\rho] \Big).
\end{split}
\end{equation}
Here the first parentheses on the right hand side is non-negative: even though $\tilde W$ is only known to be radially decreasing in $B(0,2R)$, since $\rho, \rho^*$ are both supported in $B(0,R)$, we can modify $\tilde W$ in $B(0,2R)^c$ to make it radially decreasing in $\mathbb{R}^n$ without changing the value of $\mathcal{E}_{\tilde W}[\rho], \mathcal{E}_{\tilde W}[\rho^*]$, so Riesz's rearrangement inequality yields that $\mathcal{E}_{\tilde W}[\rho^*] \geq \mathcal{E}_{\tilde W}[\rho]$.

Next we take a closer look at the second parenthesis in \eqref{temp_e}, and use the following special property of the potential $|x|^2$. 
Since the interaction energy is translational invariant, we have
\begin{equation}\label{e_quadratic}
\begin{split}
\mathcal{E}_{|x|^2}[\rho] &= \mathcal{E}_{|x|^2}[T_{x_0}\rho]\\
&= \int_{\mathbb{R}^n} \int_{\mathbb{R}^n} (T_{x_0}\rho)(x)(T_{x_0}\rho)(y) \big(|x|^2 - 2x\cdot y + |y|^2\big) dxdy\\
&= 2 \|\rho\|_1 \int_{\mathbb{R}^n} (T_{x_0}\rho)(x)|x|^2 dx  - 2 \Big|\int_{\mathbb{R}^n} (T_{x_0}\rho)(x)\, x dx\Big|^2\\
&= 2 \|\rho\|_1 M_2[T_{x_0}\rho],
\end{split}
\end{equation}
where the last equality follows from the fact that $T_{x_0}\rho = \rho(\cdot+x_0)$ has center of mass at the origin.
Note that \eqref{e_quadratic} also holds when $\rho$ is replaced by $\rho^*$ (with $x_0=0$ since $\rho^*$ is radial), so \eqref{temp_e} becomes
\[\mathcal{E}_W[\rho^*] - \mathcal{E}_W[\rho] \geq - \frac{c_{W,R}}{2}\Big(\mathcal{E}_{|x|^2}[\rho^*] - \mathcal{E}_{|x|^2}[\rho] \Big) =  c_{W,R} \|\rho\|_1 (M_2[T_{x_0}\rho] - M_2[\rho^*]),
\]finishing the proof.
\end{proof}

\section{Stability with respect to $L^1$ distance}\label{sec3}
This section is devoted to the proof of Theorem~\ref{theorem1}. As we explained in Section~\ref{sec_quadratic}, \emph{if} $\rho$ is compactly supported in some $B(0,R)$ (which is \emph{not} in the assumption of Theorem~\ref{theorem1}), it suffices to obtain the stability of the second moment. Let us first prove two stability lemmas for the second moment, one for characteristic functions and one for general densities. The proof of Theorem~\ref{theorem1} will be given after the two lemmas, and as we will see, the proof can be reduced to the compactly supported case  $\supp\rho\subset B(0,R)$ with $R=20R_*$.

We start with a preliminary lemma that gives stability of the second moment among characteristic functions. The proof follows the same idea as \cite[Lemma 1]{BC}.

  \begin{lemma}\label{lem1_1} Let $R_0>0$. For any set $D\subset \mathbb{R}^n$ with $|D|=|B(0,R_0)|$, we have
\[
 M_2[1_D] - M_2[1_{D^*}] \geq \frac{1}{2n \omega_n R_0^{n-2}}|D\Delta D^*|^2.
\]
\end{lemma}
\begin{proof}
   
    Let $\alpha:=\frac{|D\Delta D^*|}{2|D|} \in [0,1]$, so that $|D\setminus D^*| =|D^*\setminus D|  =\alpha \omega_n R_0^n$. Note that the difference of the second moment can be expressed as 
   \[
      \begin{split}
         M_2[1_D]-M_2[1_{D^*}] 
                                             & =\int_{D\setminus D^*}|x|^2dx-\int_{D^*\setminus D}|x|^2dx.
      \end{split}
   \]
  Since $|x|^2$ is a radially increasing function and $D^* = B(0,R_0)$, the first integral is minimized if $D\setminus D^*$ is an annulus with inner boundary $\partial B(0,R_0)$ and area $\alpha \omega_n R_0^n$. Likewise, the second integral is maximized if $D^*\setminus D$ is an annulus with outer boundary  $\partial B(0,R_0)$ and area $\alpha \omega_n R_0^n$. So 
   \[
      \begin{split}
         M_2[1_D]-M_2[1_{D^*}] & \ge \int_{R_0}^{R_0(1+\alpha)^{\frac{1}{n}}}n\omega_n r^{n+1}dr-\int_{R_0(1-\alpha)^{\frac{1}{n}}}^{R_0}n\omega_n r^{n+1}dr\\
                                             & =\omega_n R_0^{n+2}\int_{0}^{\alpha}((1+s)^{\frac{2}{n}}-(1-s)^{\frac{2}{n}})ds\quad \text{(by substitutions $r=R_0(1\pm s)^{\frac{1}{n}}$)}\\
                                             & =\frac{2\omega_n R_0^{n+2}}{n}\int_{0}^{\alpha}\int_{-s}^{s} (1+t)^{\frac{2}{n}-1}dtds.
      \end{split}
   \]
   Note that for all integers $n\geq 1$, $(1+t)^{\frac{2}{n}-1}$ is convex for $t\geq -1$ (and the integral domain above indeed satisfies $t\geq -1$, since $ \alpha \leq 1$). Applying Jensen's inequality to the double integral on the right hand side and using the fact that $\int_0^\alpha\int_{-s}^s dtds = \alpha^2$, we have 
   \begin{equation}\label{eq1_1}
      \begin{split}
       M_2[1_D]-M_2[1_{D^*}]  & \ge \frac{2\omega_n R_0^{n+2}}{n}\alpha^{4-\frac{4}{n}}\left(\int_{0}^{\alpha}\int_{-s}^{s} (1+t)dtds\right)^{\frac{2}{n}-1}\\
                                            & =\frac{2\omega_n R_0^{n+2}}{n}\alpha^2\\
                                            & =\frac{1}{2n \omega_n R_0^{n-2}}|D\Delta D^*|^2 \quad\text{(using $\alpha=\frac{|D\Delta D^*|}{2\omega_n R_0^n}$)}.
        \end{split}
   \end{equation}   
   \end{proof}

The next lemma deals with the stability of the second moment among all densities in $L_1^+(\mathbb{R}^n)\cap L^\infty(\mathbb{R}^n)$. 
Note that in the $n=1$ case the power of $\|\mu-\mu^*\|_1$ on the right hand side is higher, which is indeed sharp (see Remark~\ref{rmk_1d}). The $n\geq 2$ case has already been covered by a more general result by Lemou \cite[Corollary 1]{Lemou}, which deals with the stability of the $m$-th moment  $\int |x|^m \mu dx$ for $n\geq m$. We give a proof below for the sake of completeness, and also modify the proof for the $n=1$ case.

\begin{lemma}\label{lem_m2}
 Let $\mu \in L_1^+(\mathbb{R}^n)\cap L^\infty(\mathbb{R}^n)$. Assume that $M_2[\mu]<\infty$, but $\mu$ does not need to be compactly supported. Then for $n\geq 2$ we have
  \begin{equation}\label{m2_stab}
  M_2[\mu] - M_2[\mu^*]\geq \big(2n \omega_n^{\frac{2}{n}}\big)^{-1}  \|\mu\|_1^{-1+\frac{2}{n}} \|\mu\|_\infty^{-\frac{2}{n}}\, \|\mu-\mu^*\|_1^2.
  \end{equation}
  For $n=1$, \eqref{m2_stab} holds if $\mu = c1_D$ is a multiple of a characteristic function. For $n=1$ and a general density $\mu \in L_1^+(\mathbb{R})\cap L^\infty(\mathbb{R})$, we have
  \begin{equation}\label{m2_stab_n1}
   M_2[\mu] - M_2[\mu^*]\geq  \frac{1}{16} \|\mu\|_\infty^{-2}\, \|\mu-\mu^*\|_1^3.
  \end{equation}
\end{lemma}
\begin{proof}
  For $h>0$,  let $D_h:=\{x\in\mathbb{R}^n : \mu(x)>h\}$, thus $D_h^*=B(0,R_h)$ with $R_h := (\omega_n^{-1} |D_h|)^{1/n}\geq 0$. Note that $D_h$ is empty for $h>\|\mu\|_\infty$. Then we have 
   \begin{equation}\label{eq1_2_1}
      \mu(x)=\int_{0}^{\|\mu\|_\infty}1_{D_h}(x)dh,
      \end{equation} 
   and
   \begin{equation}\label{m2_rho_1}
M_2[\mu] = \int_{\mathbb{R}^n} \int_0^{\|\mu\|_\infty} 1_{D_h}(x)  |x|^2 \,dh dx = \int_0^{\|\mu\|_\infty} \int_{D_h}|x|^2  \,dx dh = \int_0^{\|\mu\|_\infty} M_2[1_{D_h}] \,dh.
\end{equation}

  By Lemma \ref{lem1_1}, we estimate the second moment difference as
   \begin{equation}\label{eq1_2_2}
      \begin{split}
         M_2[\mu]-M_2[\mu^*] & =\int_{0}^{\|\mu\|_\infty}(M_2[1_{D_h}]-M_2[1_{D_h^*}])dh
                                              \ge \frac{1}{2n \omega_n }\int_{0}^{\|\mu\|_\infty}\frac{1}{R_h^{n-2}}|D_h\Delta D_h^*|^2dh.
      \end{split}
   \end{equation}
   On the other hand, if $n\geq 2$, by (\ref{eq1_2_1}) and the Cauchy--Schwarz inequality,
   \begin{equation}\label{eq28}
      \begin{split}
         \|\mu-\mu^*\|_{1} 
                                                             & =\int_{\mathbb{R}^n}\left|\int_{0}^{\|\mu\|_\infty}(1_{D_h}(x)-1_{D_h^*}(x))dh \right| dx\\
                                                             & \leq \int_{0}^{\|\mu\|_\infty}|D_h\Delta D_h^*|dh\\
                                                             & \le \bigg(\int_{0}^{\|\mu\|_\infty}\frac{1}{R_h^{n-2}} |D_h\Delta D_h^*|^2dh\bigg)^{\frac{1}{2}} \bigg(\int_{0}^{\|\mu\|_\infty}R_h^{n-2}dh\bigg)^{\frac{1}{2}} . 
      \end{split}
   \end{equation}
   To control the last integral on the right hand side, note that $ \int_{0}^{\|\mu\|_\infty} \omega_n R_h^n dh = \int_{0}^{\|\mu\|_\infty}  |D_h| dh = \|\mu\|_1$, thus H\"older's inequality (and the assumption $n\geq 2$) gives
   \[
   \int_{0}^{\|\mu\|_\infty}R_h^{n-2}dh \leq \bigg(\int_{0}^{\|\mu\|_\infty}R_h^n dh\bigg )^{\frac{n-2}{n}} \|\mu\|_\infty^{\frac{2}{n}} = \omega_n^{-\frac{n-2}{n}}\|\mu\|_1^{\frac{n-2}{n}} \|\mu\|_\infty^{\frac{2}{n}}.
   \]
   Plugging this into \eqref{eq28} and combining it with \eqref{eq1_2_2}, we have the inequality \eqref{m2_stab} for $n\geq 2$.   

It remains to deal with the $n=1$ case. In the special case that $\mu=c 1_D$ is a multiple of a characteristic function, we have $R_h \equiv \|\mu\|_1/(2\|\mu\|_\infty)$ for all $h\in[0,\|\mu\|_\infty)$, and applying it to \eqref{eq28} would still yield \eqref{m2_stab}. For a general density, using that $R_h = \frac{1}{2}|D_h| \geq \frac{1}{4} |D_h\Delta D_h^*|$, \eqref{eq1_2_2} becomes 
\begin{equation}\label{eq210}
M_2[\mu]-M_2[\mu^*] \geq \frac{1}{4} \int_{0}^{\|\mu\|_\infty} R_h |D_h\Delta D_h^*|^2dh \geq \frac{1}{16} \int_{0}^{\|\mu\|_\infty}  |D_h\Delta D_h^*|^3dh,
\end{equation}
thus we can proceed as the first two steps of \eqref{eq28} and use the H\"older equality to obtain
\begin{equation}\label{eq29}
 \|\mu-\mu^*\|_{L^1(\mathbb{R})}\leq  \int_{0}^{\|\mu\|_\infty}|D_h\Delta D_h^*|dh \leq \|\mu\|_{\infty}^{\frac23}\bigg(\int_{0}^{\|\mu\|_\infty}|D_h\Delta D_h^*|^3dh\bigg)^{\frac13},
\end{equation}
and combining it with \eqref{eq210} yields the inequality \eqref{m2_stab_n1} for the $n=1$ case, finishing the proof.
\end{proof}

\begin{remark}\label{rmk_1d}
In the $n=1$ case, for a general density $\mu$, the following example shows that the power 3 on the right hand side of \eqref{m2_stab_n1} is indeed sharp. For $\epsilon\ll 1$, let $\mu = 1_{[-1,1]} + 1_{[0,2\epsilon]}$, so that $\mu^* = 1_{[-1,1]} + 1_{[-\epsilon,\epsilon]}$. One can easily check that $\|\mu\|_1, \|\mu\|_\infty \sim 1$, whereas $M_2[\mu]-M_2[\mu^*] \sim \epsilon^3$ and $\|\mu-\mu^*\|_1 \sim \epsilon$. As a result, one can only expect
\[
M_2[\mu]-M_2[\mu^*] \geq C(\|\mu\|_1, \|\mu\|_\infty) \|\mu-\mu^*\|_1^3,
\]
where the power 3 cannot be lowered.
\end{remark}

Now we are ready to prove Theorem~\ref{theorem1}.

\begin{proof}[\textbf{\textup{Proof of Theorem~\ref{theorem1}}}] Throughout this proof let us fix $R := 20R_*$, and we aim to show that the proof can be reduced to the case where $\rho$ is supported in $B(0,R)$.

Let us define a radially decreasing $\phi \in C(\mathbb{R}^n)$ as
\[\phi(x) := \begin{cases} -|x|^2 &\text{ for } |x|<2R\\
-(2R)^2 &\text{ for } |x|\geq 2R.\end{cases}
\]
We then define 
\[\tilde V := W-\frac{c_{W,R}}{2}\phi,
\] 
where $c_{W,R}$ is given by \eqref{c_general}, and it becomes \eqref{eq:cwkr} for the potentials $W_k$. Note that $\tilde V$ coincides with the function $\tilde W$ given by \eqref{def_tildeW} in $B(0,2R)$, thus the second paragraph of the proof of Lemma~\ref{lem_reduce_m2} yields that $\tilde V$ is radially decreasing in $B(0,2R)$. In fact, $\tilde V$ is radially decreasing in $\mathbb{R}^n$, since $\phi \in C(\mathbb{R}^n)$ and $\phi=\text{const}$ for $|x|\geq 2R$, and $W$ is radially decreasing in $\mathbb{R}^n$. (As a contrast, $\tilde W$ might not be radially decreasing outside $B(0,2R)$.) Decomposing $W$ as $\tilde V +\frac{c_{W,R}}{2}\phi$, we have
\begin{equation}\label{eq_e_eta}
\begin{split}
\mathcal{E}_W[\rho^*] - \mathcal{E}_W[\rho] 
& = (\mathcal{E}_{\tilde V}[\rho^*] - \mathcal{E}_{\tilde V}[\rho] )+ \frac{c_{W,R}}{2} (\mathcal{E}_{\phi}[\rho^*] - \mathcal{E}_{\phi}[\rho] )\\
&\geq \frac{c_{W,R}}{2} (\mathcal{E}_{\phi}[\rho^*] - \mathcal{E}_{\phi}[\rho] ),
\end{split}
\end{equation}
where the inequality follows from the Riesz's rearrangement inequality as well as the fact that $\tilde V$ is radially decreasing in $\mathbb{R}^n$.

Throughout the rest of the proof let \[
m_0 := \sup_{x\in\mathbb{R}^n} \int_{|y|<5R_*} (T_x\rho)(y) dy.
\] Below we discuss the following two cases, and in each case we aim to obtain a lower bound of $\mathcal{E}_{\phi}[\rho^*] - \mathcal{E}_{\phi}[\rho]$.

\noindent\textbf{Case 1.} $m_0 \leq \frac{9}{10}\|\rho\|_1$. In this case any ball with radius $5R_*$ misses at least $\frac{1}{10}$ of the mass of $\rho$. Thus for any $x\in\mathbb{R}^n$ we have
\[
(\rho*\phi)(x) = \int_{\mathbb{R}^n} \rho(x+y)\phi(y)dy \leq -\int_{|y|\geq 5R_*} T_x \rho(y) (5R_*)^2 dy \leq -\frac{5}{2}\|\rho\|_1 R_*^2,
\]
where the first step is due to $\phi(y)=\phi(-y)$; the second step follows from $\phi\leq 0$, and $\phi(y)\leq -(5R_*)^2$ for $|y|\geq 5R_*$; and in the last step we use the assumption $m_0 \leq \frac{9}{10}\|\rho\|_1$.
This directly gives \[
\mathcal{E}_{\phi}[\rho] = \int_{\mathbb{R}^n} \rho(x) (\rho*\phi)(x) dx \leq -\frac{5}{2}\|\rho\|_1^2 R_*^2.\]
As a contrast, since $\supp \rho^* \subset B(0,R_*)$, we have
\[
\mathcal{E}_{\phi}[\rho^*] = -\mathcal{E}_{|x|^2}[\rho^*] = -2\|\rho\|_1 M_2[\rho^*] \geq -2\|\rho\|_1^2 R_*^2,
\]
where we used the definition of $\phi$ in the first equality, and applied \eqref{e_quadratic} to $\rho^*$ to get the second equality.
The above two inequalities immediately yield
\begin{equation}\label{case1}
\mathcal{E}_{\phi}[\rho^*] - \mathcal{E}_{\phi}[\rho] \geq \frac{1}{2}  \|\rho\|_1^2 R_*^2 \geq \frac{1}{2}  \|\rho\|_1^2 R_*^2 \,\delta(\rho)^2,
\end{equation}
where the last inequality follows from $\delta(\rho) \in (0,1)$.

\noindent\textbf{Case 2.} $m_0 > \frac{9}{10}\|\rho\|_1$. In this case there exists $a\in\mathbb{R}^n$, such that $\int_{B(a,5R_*)} \rho\, dy > \frac{9}{10}\|\rho\|_1$. 

Let us construct a new density $\tilde \rho$ supported in $ B(a,20R_*)$ that has the same distribution as $\rho$. Let $T:\supp\rho \to \mathbb{R}^n$ be a measure-preserving map, such that $T = id$ in $\supp\rho \cap B(a,20 R_*)$, and $T(y) \subset B(a,5R_*)$ for $y\in \supp\rho \cap B(a,20 R_*)^c$. (Since $|\supp\rho| = |B(0,R_*)|$, there is enough room in $B(a,5R_*)$ for the map. Note that $T$ does not need to be continuous.) We then define $\tilde \rho(x) := \rho(T^{-1}(x))$. See Figure~\ref{fig_density} for an illustration of the supports of $\rho$ and $\tilde \rho$.

\begin{figure}[h!]
\begin{center}
\includegraphics[scale=0.8]{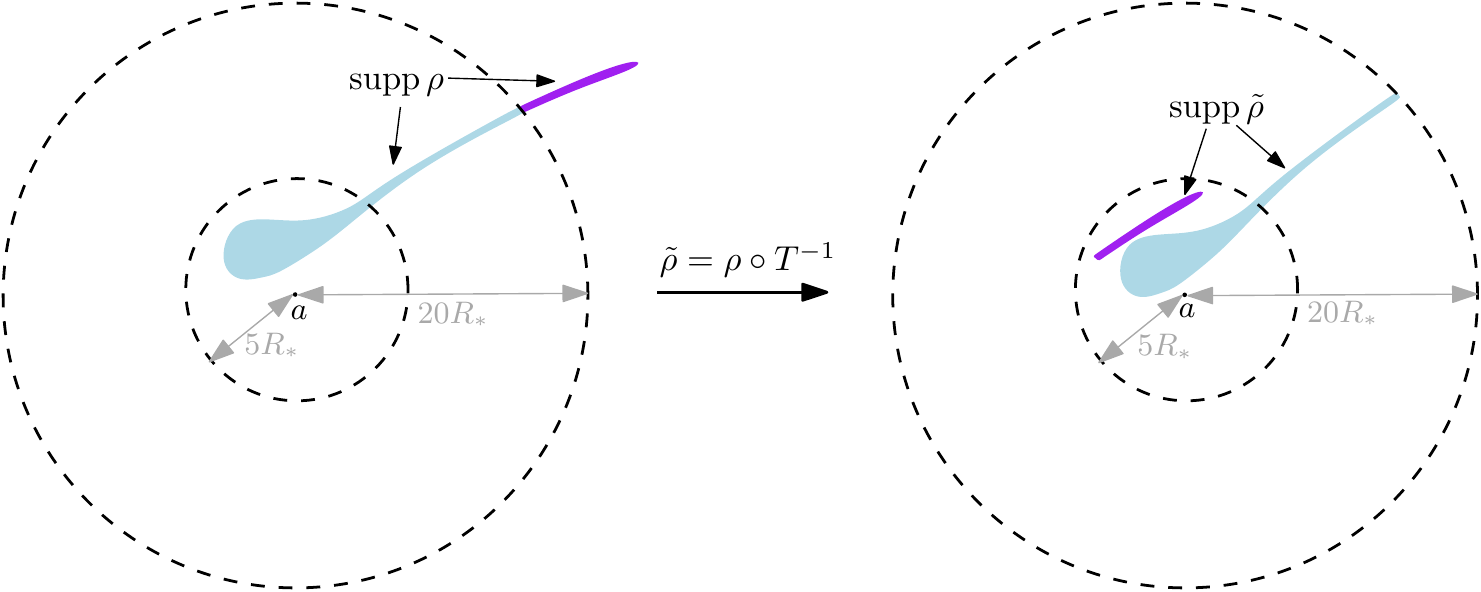}
\caption{The left figure shows $\supp\rho$, colored into blue and purple depending on whether it lies inside $B(a,20R_*)$. The right  figure shows $\supp\tilde\rho$, where the blue set remain unchanged, whereas the purple set is mapped inside $B(a,5R_*)$. 
\label{fig_density}}
\end{center}
\end{figure}

Using $\tilde\rho(x)=\rho(T^{-1}(x))$ and the fact that $T$ is measure preserving, a change of variables gives
\[
\mathcal{E}_{\phi}[\tilde\rho] = \iint_{\mathbb{R}^n \times \mathbb{R}^n} \rho(T^{-1}(x))\rho(T^{-1}(y)) \phi(x-y) dx dy =  \iint_{\mathbb{R}^n \times \mathbb{R}^n} \rho(x)\rho(y) \phi(T(x)-T(y)) dxdy,
\]
leading to
\begin{equation}\label{double_int}
\mathcal{E}_{\phi}[\tilde\rho]  - \mathcal{E}_{\phi}[\rho] =   \int_{\supp\rho}\int_{\supp\rho} \rho(x)\rho(y) \underbrace{\Big(\phi\big(T(x)-T(y)\big) - \phi\big(x-y\big)\Big)}_{=:K(x,y)} dxdy.
\end{equation}

Let us break $\supp\rho$ into the union of the following three sets  $A_{in} := (\supp\rho)\cap B(a,5R_*)$, $A_{mid} := (\supp\rho) \cap B(a,20R_*) \setminus B(a,5R_*)$ and $A_{out} :=  (\supp\rho)\cap B(a,20R_*)^c$. The assumption of Case 2 and our choice of $a$ yield $\int_{A_{in}}\rho dx > \frac{9}{10}\|\rho\|_1$, thus $\int_{A_{mid} \cup A_{out}}\rho dx \leq \frac{1}{10}\|\rho\|_1$.  

Let $m_{out} := \int_{A_{out}}\rho dx \geq 0$.  We claim that 
\begin{equation}\label{claim_tilde_rho}
\mathcal{E}_{\phi}[\tilde\rho]  - \mathcal{E}_{\phi}[\rho] \geq 10\, m_{out} \|\rho\|_1 R_*^2.
\end{equation} To show this, we decompose the integral domain $\supp\rho \times \supp\rho$ in \eqref{double_int}  into the disjoint union of the following sets in (a)--(f):

\noindent(a) If $x, y \in A_{in} \cup A_{mid}$, then $K(x,y)=0$ since $T=id$ in this set. Thus this set gives no contribution to the double integral.

\noindent(b) If $x\in A_{out}$ and $y\in A_{in}$, then $T(x),T(y) \subset B(a,5R_*)$, thus $|T(x)-T(y)| < 10R_*$. Since $|x-y|>15R_*$, it leads to $K(x,y) \geq (-10^2 + 15^2) R_*^2$. Thus 
\[
\iint_{A_{out} \times A_{in}} \rho(x)\rho(y)K(x,y) dxdy \geq \frac{9}{10}(15^2 - 10^2)  m_{out} \|\rho\|_1 R_*^2 > 100 m_{out} \|\rho\|_1 R_*^2 .
\]

\noindent(c) Clearly the same estimate in (b) also holds for $\iint_{A_{in} \times A_{out}} \rho(x)\rho(y) K(x,y) dxdy$.

\noindent(d) If $x\in A_{out}$ and $y\in A_{mid}$,  we have $T(x)\in B(a,5R_*)$ and $T(y)=y \in B(a,20R_*)$, thus $|T(x)-T(y)| \leq 25 R_*$. Using this with $\phi(x-y)\leq 0$ gives $K(x,y) \geq -25^2 R_*^2$, and combining it with $\int_{A_{mid}}\rho dy \leq \frac{1}{10}\|\rho\|_1$ yields
\[
\iint_{A_{out} \times A_{mid}} \rho(x)\rho(y) K(x,y) dxdy \geq -\frac{1}{10}(25)^2 m_{out} \|\rho\|_1 R_*^2 > -70m_{out} \|\rho\|_1 R_*^2 .
\]

\noindent(e) Clearly the same estimate in (d) also holds for $\iint_{A_{mid} \times A_{out}} \rho(x)\rho(y) K(x,y) dxdy$.

\noindent(f) If $x,y \in A_{out}$, we have $T(x),T(y) \in B(a,5 R_*)$, thus $K(x,y) \geq -10^2 R_*^2$. Combining it with $\phi\leq 0$ and $m_{out} = \int_{A_{out}}\rho dy \leq \frac{1}{10}\|\rho\|_1$ yields

\[
\iint_{A_{out} \times A_{out}} \rho(x)\rho(y) K(x,y) dxdy \geq -\frac{1}{10}(10)^2 m_{out} \|\rho\|_1 R_*^2  = -10m_{out} \|\rho\|_1 R_*^2 .
\]

We then obtain the claim \eqref{claim_tilde_rho} by adding the estimates in (a)--(f) together. 

Recall that $\tilde \rho^* = \rho^*$ since $T$ is measure-preserving. This implies $\mathcal{E}_{\phi}[\rho^*] = \mathcal{E}_{\phi}[(\tilde \rho)^*]\geq \mathcal{E}_{\phi}[\tilde\rho]$, thus
 \eqref{claim_tilde_rho} directly gives
\begin{equation}\label{e_phi_tilde}
\mathcal{E}_{\phi}[\rho^*]  - \mathcal{E}_{\phi}[\rho] \geq \mathcal{E}_{\phi}[\tilde\rho]  - \mathcal{E}_{\phi}[\rho] \geq 10 \,m_{out} \|\rho\|_1 R_*^2.
\end{equation}

As a result, Case 2 can be divided into the following two sub-cases:

\noindent \underline{Case 2.1.} $m_{out} > \frac{1}{10} \delta(\rho) \|\rho\|_1$. In this case we directly use \eqref{e_phi_tilde} to conclude that
\begin{equation}\label{case21}
\mathcal{E}_{\phi}[\rho^*]  - \mathcal{E}_{\phi}[\rho] \geq \delta(\rho)  \|\rho\|_1^2 R_*^2 \geq   \|\rho\|_1^2 R_*^2\,\delta(\rho)^2,
\end{equation}
where we used $\delta(\rho)\in[0,1)$ in the last step.

\noindent \underline{Case 2.2.} $m_{out} \leq  \frac{1}{10} \delta(\rho) \|\rho\|_1$. In this case we claim that $\delta(\tilde\rho) \geq \frac{9}{10}\delta(\rho)$. To show this, note that the assumption in Case 2.2 gives $\|\rho-\tilde\rho\|_1 = 2m_{out} \leq \frac{1}{5}\delta(\rho)\|\rho\|_1$. Hence for any $b\in \mathbb{R}^n$, triangle inequality and the fact that $(\tilde\rho)^* = \rho^*$ yield
\[
\|T_b \tilde\rho - (\tilde \rho)^*\|_1 \geq \|T_b \rho - \rho^*\|_1 - \|T_b (\tilde \rho - \rho)\|_1 \geq 2\delta(\rho)\|\rho\|_1-  \frac{1}{5} \delta(\rho) \|\rho\|_1 = \frac{9}{5}\delta(\rho)\|\rho\|_1.
\]
Taking the infimum in $b$ in the above inequality and dividing by $2\|\tilde\rho\|_1$ (and note that $\|\rho\|_1 = \|\tilde\rho\|_1$) yields the claim.

Using \eqref{claim_tilde_rho} and the fact that $\rho^* = (\tilde\rho)^*$, we have
\begin{equation}\label{diff_phi}
\begin{split}
\mathcal{E}_\phi[\rho^*] - \mathcal{E}_\phi[ \rho] &\geq 
\mathcal{E}_\phi[(\tilde\rho)^*] - \mathcal{E}_\phi[\tilde \rho]\\
& = \mathcal{E}_{|x|^2}[\tilde \rho] - \mathcal{E}_{|x|^2}[(\tilde\rho)^*] \\&= 2\|\rho\|_1 ( M_2[T_{x_0}\tilde \rho] - M_2[(\tilde\rho)^*]),
\end{split}
\end{equation}
where $x_0$ is the center of mass of $\tilde\rho$.  
Here the second step follows from the facts that $\supp \tilde\rho \subset B(0,R)$ (recall that $R = 20R_*$) and $\phi = -|x|^2$ in $B(0,2R)$, and the last step follows from the identity \eqref{e_quadratic} applied to $\tilde\rho$ and $(\tilde\rho)^*$.

Next we will apply Lemma~\ref{lem_m2} to estimate the second moment difference on the right hand side of \eqref{diff_phi}.  If $n\geq 2$, we apply \eqref{m2_stab}  to $\mu = T_{x_0}\tilde \rho$ to obtain 
 \begin{equation}\label{m2_temp1}
 M_2[T_{x_0}\tilde \rho] - M_2[(\tilde\rho)^*] \geq c(n)   \|\rho\|_1^{1+\frac{2}{n}} \|\rho\|_\infty^{-\frac{2}{n}} \,\delta(\tilde\rho)^2,
 \end{equation}
 where we used $\|\mu-\mu^*\|_1 \geq 2\|\mu\|_1 \delta(\mu) = 2\|\rho\|_1\delta(\mu)$ and $\delta(\mu) = \delta(\tilde\rho)$.
 Plugging \eqref{m2_temp1} into \eqref{diff_phi}, and using $\delta(\tilde\rho) \geq \frac{9}{10}\delta(\rho)$, we have 
 \begin{equation}\label{case22}
 \mathcal{E}_\phi[\rho^*] - \mathcal{E}_\phi[ \rho] \geq c(n) \|\rho\|_1^{2+\frac{2}{n}} \|\rho\|_\infty^{-\frac{2}{n}} \,\delta(\rho)^2.
 \end{equation}
 
 Note that the two inequalities \eqref{case1} and \eqref{case21} from Case 1 and Case 2.1 are both stronger than \eqref{case22}, due to the relation $\omega_n \|\rho\|_\infty R_*^n \geq \|\rho\|_1$. As a result, in all the three cases we have \eqref{case22} for $n\geq 2$. Plugging it into \eqref{eq_e_eta} finally gives
 \begin{equation}\label{temp_n2}
 \mathcal{E}_W[\rho^*] - \mathcal{E}_W[\rho] \geq c(n) c_{W,R} \|\rho\|_1^{2+\frac{2}{n}} \|\rho\|_\infty^{-\frac{2}{n}} \,\delta(\rho)^2,
 \end{equation}
 finishing the proof for $n\geq 2$. 
 
 If $n=1$ and $\rho$ is a multiple of characteristic function (note that $\tilde \rho$ is also a multiple of characteristic function due to our construction), we can still apply \eqref{m2_stab} in Lemma~\ref{lem_m2} to obtain \eqref{m2_temp1}, therefore \eqref{temp_n2} still holds in this case.
 
If $n=1$ and $\rho$ is a general density, we have to apply \eqref{m2_stab_n1} in Lemma~\ref{lem_m2}, thus instead of \eqref{m2_temp1} we now have
\[
 M_2[T_{x_0}\tilde \rho] - M_2[(\tilde\rho)^*] \geq c   \|\rho\|_1^3\, \|\rho\|_\infty^{-2} \,\delta(\tilde\rho)^3.
 \]
This is the same as \eqref{m2_temp1} except with $\delta(\tilde\rho)^2$ replaced by $\delta(\tilde\rho)^3$, thus 
 an identical argument as \eqref{case22}--\eqref{temp_n2} now gives \eqref{temp_n2} with $\delta(\rho)^2$ replaced by $\delta(\rho)^3$, finishing the proof for $n=1$.
  \end{proof}

\begin{remark}\label{rmk_1d_2}
If $n=1$, for a general density $\rho$, the following example shows that the power of $\delta(\rho)$ on the right hand side of \eqref{thm1_eq1} cannot be lower than 3, thus the power 3 in Theorem~\ref{theorem1} is indeed sharp. For $\epsilon\ll 1$, let $\rho := 1_{[-1,1]} + 1_{[0,2\epsilon]}$ be the same as in Remark~\ref{rmk_1d}, so $\rho^* = 1_{[-1,1]} + 1_{[-\epsilon,\epsilon]}$. We can easily check that $\|\rho\|_1, \|\rho\|_\infty \sim 1$, and $\delta(\rho)\sim\epsilon$. In addition, we claim that
\[
\mathcal{E}_W[\rho^*]- \mathcal{E}_W[\rho] \leq C(W) \epsilon^3,
\]
which would imply the sharpness of $\delta(\rho)^3$ in \eqref{thm1_eq1} for $n=1$.

To prove the claim, let $f_1 := 1_{[-1,1]} $ and $f_2 := 1_{[-\epsilon,\epsilon]}$, so $\rho^* = f_1+f_2$, and $\rho = f_1 + T_{-\epsilon} f_2$. Thus
\begin{equation}\label{temp_dim1}
\begin{split}
\mathcal{E}_W[\rho^*]- \mathcal{E}_W[\rho] &= \int_{\mathbb{R}} (f_1+f_2) ((f_1+f_2)*W) dx - \int_{\mathbb{R}} (f_1+T_{-\epsilon} f_2) ((f_1+T_{-\epsilon} f_2)*W) dx\\
&= 2 \int_{\mathbb{R}} (f_2 - T_{-\epsilon} f_2) (f_1*W) dx\\
& = 2\left( \int_{-\epsilon}^0 (f_1*W)(x) dx -  \int_{\epsilon}^{2\epsilon} (f_1*W)(x) dx\right),
\end{split}
\end{equation}
where the second equality follows from the identities $\int f_1(f_2*W) dx = \int f_2(f_1*W) dx$ and $\int (T_{-\epsilon}  f_2) (T_{-\epsilon} f_2)*W dx= \int   f_2 ( f_2*W ) dx$. Note that $(f_1*W)(x) = \int_{-1+x}^{1+x} W(y)dy$ is radially decreasing, and $f_1*W \in C^\infty((-1,1))$ since $W \in C^\infty(\mathbb{R}\setminus \{0\})$. Thus $\sup_{[-\frac{1}{2},\frac{1}{2}]} |(f_1*W)''|<C(W)$. Combining it with the symmetry of $f_1*W$ gives 
\[
(f_1*W)(0)-(f_1*W)(x) \leq C(W) x^2\quad\text{ for all }|x|<\frac{1}{2}.
\] Applying this to the right hand side of \eqref{temp_dim1} yields
\[
\mathcal{E}_W[\rho^*]- \mathcal{E}_W[\rho] \leq 2\epsilon((f_1*W)(0)-(f_1*W)(2\epsilon)) \leq C(W)\epsilon^3,
\]
finishing the proof of the claim.
\end{remark}

\section{Stability with respect to 2-Wasserstein distance}\label{sec4}

In this section we aim to prove Theorem~\ref{theorem2}, which is a stability estimate of the interaction energy with respect to the 2-Wasserstein distance.
Since $\supp\rho\subset B(0,R)$ by assumption, using Lemma~\ref{lem_reduce_m2}, it suffices to prove the stability of the second moment with respect to the 2-Wasserstein distance. To our best knowledge, we are unaware of such result in the literature. Below we state and prove such an estimate, which might be of independent interest. 

\begin{proposition}\label{prop_m2}
For any $\rho \in \mathcal{P}_2(\mathbb{R}^n)$, the following inequality holds:
\begin{equation}\label{m2_goal}
M_2[\rho] - M_2[\rho^*] \geq W_2^2(\rho, \rho^*).
\end{equation}
\end{proposition}

\begin{remark}\label{rmk42}
One can easily check that the power 2 and the constant 1 on the right hand side of \eqref{m2_goal} are both sharp: let $B\subset\mathbb{R}^n$ be a ball centered at the origin with $|B|=1$, and let $D_\epsilon$ be any subset of $B$ with volume $\epsilon$. For $R>3$, let $x_R = (R,0,\dots, 0)$, and finally let $\rho = 1_{B \setminus D_\epsilon} + T_{x_R} 1_{D_\epsilon}$. (Note that $\rho^* = 1_B$.) Then we have $M_2[\rho] - M_2[\rho^*] \leq \epsilon (R+2)^2$ and $W_2^2(\rho, \rho^*)\geq \epsilon (R-2)^2$. By fixing an $R>3$ and sending $\epsilon\to 0^+$, we know that the power 2 in \eqref{m2_goal} is sharp. Since $R$ can be chosen as arbitrarily large,  the constant 1 in \eqref{m2_goal} is also sharp.
\end{remark}

Before proving Proposition~\ref{prop_m2}, let us first introduce some preliminary results on optimal transport, which will be used in the proof.  For any two density functions $\rho_1, \rho_2 \in L_+^1(\mathbb{R}^n)$ with the \emph{same} integral $\int_{\mathbb{R}^n} \rho_1 dx = \int_{\mathbb{R}^n} \rho_2 dx = M$, 
if $M_2[\rho_i]<\infty$ for $i=1,2$,  optimal transport theory (see \cite[Section 2]{Villani} for example) shows that the infimum
\[
\inf\Big\{\int_{\mathbb{R}^n} \rho_1(x) |T(x)-x|^2 dx: T \# \rho_1 = \rho_2 \Big\}
\]
can be achieved by some map $T:\mathbb{R}^n\to \mathbb{R}^n$. (The proof is done for probability densities with $M=1$, but the same proof indeed works for all $M>0$.) For any $M>0$, with a slight abuse of notation, we will call such $T$ the \emph{optimal transport map} between $\rho_1$ and $\rho_2$.

The following lemma shows that Proposition~\ref{prop_m2} is true among characteristic functions.

\begin{lemma}\label{lemma_char}
Assume $D \subset \mathbb{R}^n$ satisfies that $M_2[1_D]<\infty$.  Then we have 
\begin{equation}\label{lemma_goal}
M_2[1_D] - M_2[1_D^*] \geq \inf\Big\{\int_{D^*} |T(x)-x|^2 dx: T \# 1_D^* = 1_D \Big\}.
\end{equation}
\end{lemma}

\begin{proof}  By the discussion before this lemma, we know there exists some optimal transport map $T:\mathbb{R}^n\to \mathbb{R}^n$ with $T\# 1_{D^*} = 1_D$, such that the infimum on the right hand side \eqref{lemma_goal} is achieved by $T$.  For $t\in[0,1]$, let  $\mu_t$ be given by
 \[
\mu_t := ((1-t) id + tT)\# 1_D^*.\] Note that $\mu_0=1_D^*$ and $\mu_1 = 1_D$. (If $|D|=1$, $\{\mu_t\}_{t\in[0,1]}$ is the geodesics connecting $1_D^*$ and $1_D$ in 2-Wasserstein metric.) By definition of $\mu_t$ and using the property of the push-forward map (see \cite[Eq (5.2.2)]{AGS}), we have
\begin{equation}\label{eq:2nd_der}
\begin{split}
M_2[\mu_t] &= \int_{\mathbb{R}^n} \mu_t(x) |x|^2  dx= \int_{\mathbb{R}^n} 1_D^*(x) \big|(1-t)x + tT(x)\big|^2  dx\\
&= M_2[1_D^*] + 2t \int_{D^*} x\cdot(T(x)-x) dx + t^2 \int_{D^*}  |T(x)-x|^2  dx
\end{split}
\end{equation}
for all $t\in[0,1]$, thus the function $t\mapsto M_2[\mu_t]$ is in $C^2([0,1])$. 

Next we claim that $M_2[\mu_t] \geq M_2[\mu_0]$ for all $t\in[0,1]$.  To see this, note that for any $t\in [0,1]$, $\mu_t$ is a nonnegative density with integral $|D|$. Using that $T$ is the optimal map such that $T\#\mu_0=\mu_1$, the function $t\mapsto \int_{\mathbb{R}^n} \mu_t^p dx$ is convex for $t\in[0,1]$ for any $1<p<\infty$ (see \cite[Proposition 9.3.9]{AGS} or \cite[Theorem 2.2]{McCann}), which leads to
\[
\|\mu_t\|_p \leq \max\{\|\mu_0\|_p, \|\mu_1\|_p\} = \|1_D\|_p \quad\text{ for all }1<p<\infty.
\] Sending $p\to\infty$ in the above expression gives $\|\mu_t\|_\infty \leq 1$ for all $t\in[0,1]$. It is easy to see that among all functions $g \in L^1_+(\mathbb{R}^n) \cap L^\infty(\mathbb{R}^n)$ satisfying $\|g\|_1 = |D|$ and $\|g\|_\infty\leq 1$ ($g$ does not need to be a characteristic function), $\mu_0 = 1_D^*$ is the one that minimizes the second moment, due to the fact that $|x|^2$ is a radially increasing function. This finishes the proof of the claim.

As a result of the claim, we have $\frac{d^+}{dt} M_2[\mu_t]\big|_{t=0}\geq 0$ (where $\frac{d^+}{dt}$ denotes the right derivative). Combining it with \eqref{eq:2nd_der} gives that $\int_{D^*}  x\cdot(T(x)-x) dx\geq 0$, thus \eqref{eq:2nd_der}  becomes
\[
M_2[\mu_t] \geq M_2[1_D^*] + t^2  \int_{D^*}  |T(x)-x|^2  dx,
\]
and plugging in $t=1$ gives the desired result.
\end{proof}

Now we are ready to prove Proposition~\ref{prop_m2} for a general probability density.\footnote{One might be tempted to use the same idea as in Lemma~\ref{lemma_char} and let $\mu_t$ be the geodesics  connecting $\rho^*$ and $\rho$ in $W_2$ metric. Although \eqref{eq:2nd_der} still holds with $1_{D^*}$ replaced by $\rho^*$, it is unclear to us whether $M_2[\mu_t]\geq M_2[\mu_0]$ for $t\in[0,1]$ is true for a general density. (Even though $\|\mu_t\|_p \leq \|\rho^*\|_p$ for $1<p\leq\infty$ is still true, this does not imply $M_2[\mu_t] \geq M_2[\rho^*]$ since $\rho^*$ is no longer a characteristic function.) To circumvent this difficulty, the proof of Proposition~\ref{prop_m2} does not use the optimal  transport map between $\rho^*$ and $\rho$. Instead, we will decompose $\rho, \rho^*$ using the layer-cake formula, and build a (non-optimal) transport plan by integrating the optimal map for each layer.}

\begin{proof}[\textbf{\textup{Proof of Proposition~\ref{prop_m2}}}]
For any $h>0$, let $D_h := \{x\in\mathbb{R}^n: \rho(x)>h\}$. Then we have that $\rho(x) = \int_0^\infty 1_{D_h}(x) dh$ for every $x$, thus the same computation as \eqref{m2_rho_1} gives
\begin{equation}\label{m2_rho}
M_2[\rho] 
 = \int_0^\infty M_2[1_{D_h}] \,dh.
\end{equation}
Note that $M_2[1_{D_h}]$ is decreasing in $h$, since the set $D_h$ is decreasing in $h$ (in the sense that $D_{h_1} \subset D_{h_2}$ for $h_1>h_2>0$). 
By assumption we have $M_2[\rho]<\infty$, and by \eqref{m2_rho} this leads to $M_2[1_{D_h}] < \infty$ for all $h>0$. 

For any $h>0$, let $T_h: \mathbb{R}^n \to \mathbb{R}^n$ be the optimal transport map such that $T_h \# 1_{D_h^*} = 1_{D_h}$,  and let $\gamma_h := (id \times T_h) \# 1_{D_h^*}$ be the transport plan corresponding to $T_h$. Note that $\gamma_h$ is a measure on $\mathbb{R}^n\times \mathbb{R}^n$ with first marginal $1_{D_h^*}$ and second marginal $1_{D_h}$.

By Lemma~\ref{lemma_char} and the optimality of $T_h$, we have
\begin{equation}\label{ineq_layer}
M_2[1_{D_h}] - M_2[1_{D_h^*}] \geq \int_{D_h^*} |T_h(x)-x|^2 dx = \iint_{\mathbb{R}^n\times \mathbb{R}^n} |x-y|^2 d\gamma_h(x,y).
\end{equation}
Finally, let
\[
\gamma(x,y) := \int_0^\infty \gamma_h(x,y) dh,
\]
which is a measure on $\mathbb{R}^n\times \mathbb{R}^n$ with first marginal $\rho^*$ and second marginal $\rho$. Therefore $\gamma$ is a transport plan between $\rho^*$ and $\rho$, but not necessarily an optimal one.

Integrating \eqref{ineq_layer} in $h$ and applying \eqref{m2_rho} gives
\[
M_2[\rho] - M_2[\rho^*] = \int_0^\infty \big(M_2[1_{D_h}] - M_2[1_{D_h^*}] \big) dh\geq   \iint_{\mathbb{R}^n\times \mathbb{R}^n} |x-y|^2 d\gamma(x,y).
\]
Using that $\gamma$ is a transport plan between $\rho^*$ and $\rho$ (which is not necessarily optimal), it gives $ \iint_{\mathbb{R}^n\times \mathbb{R}^n} |x-y|^2 d\gamma(x,y) \geq W_2^2(\rho,\rho^*)$, and plugging it into the above inequality yields
\[
M_2[\rho] - M_2[\rho^*] \geq W_2^2(\rho, \rho^*),
\]
thus we can conclude.
\end{proof}

Once we obtain Proposition~\ref{prop_m2}, Theorem~\ref{theorem2} follows as a direct consequence. 
\begin{proof}[\textbf{\textup{Proof of Theorem~\ref{theorem2}}}]
Using Lemma~\ref{lem_reduce_m2} and the fact that $\|\rho\|_1=1$, we have
\[
\mathcal{E}_W[\rho^*] - \mathcal{E}_W[\rho] \geq c_{W,R} (M_2[T_{x_0}\rho] - M_2[\rho^*]).
\]
Here $c_{W,R}$ is given by \eqref{c_general} for a general $W$, and for potentials $W_k$ with $k\in (-n,2]$ it becomes $c_{W_k,R} = (2R)^{k-2}$ by \eqref{eq:cwkr}.

Applying Proposition~\ref{prop_m2} (with $\rho$ replaced by $T_{x_0}\rho$; note that $\rho^* = (T_{x_0}\rho)^*$), we have
\[
M_2[T_{x_0}\rho] - M_2[\rho^*] \geq W_2^2(T_{x_0}\rho, \rho^*).
\]
We then conclude the proof by combining the above two inequalities together.
\end{proof}

\begin{remark}\label{rmk_w2}(a)
For potentials $W_k$ with $k\in(-n,2]$, the following scaling argument shows that the power $k-2$ in the constant $c(W_k,R) = (2R)^{k-2}$ in Theorem~\ref{theorem2} is sharp. Fix any $\rho\in\mathcal{P}(\mathbb{R}^n)\cap L^\infty(\mathbb{R}^n)$ supported in $B(0,1)$ that does not coincide with $T_{-x_0}\rho^*$, and denote 
$A := \mathcal{E}_{W_k}[\rho^*] - \mathcal{E}_{W_k}[\rho]>0$ and $B := W_2^2(T_{x_0}\rho, \rho^*)>0$. For any $R>0$, define $\rho_R := R^{-n} \rho(R^{-1} x)$, so that $\rho_R$ is supported in $B(0,R)$, and it has center of mass $x_{0R} := Rx_0$. One can easily check that $\mathcal{E}_{W_k}[\rho_R^*] - \mathcal{E}_{W_k}[\rho_R] = R^k A$, and $W_2^2(T_{x_{0R}}\rho_R, \rho_R^*) = R^2B $. In order for $R^k A \geq C(W_k,R) R^2 B$ to hold for all $R>0$, we know that $C(W_k,R) \sim R^{k-2}$ is the sharp power.

(b) We point out that for any potential $W$ with $\lim_{|x|\to\infty} W(x) |x|^{-2}=0$,  it is necessary to allow $C(W,R)$ to depend on $R$. Let $B\subset\mathbb{R}^n$ be a ball centered at the origin with $|B|=1$, and let $B_+ := B\cap \{x_1\geq 0\}, B_- := B\cap \{x_1< 0\}$. For $R>3$, let $x_R = (R,0,\dots, 0)$, and finally let $\rho =T_{x_R} B_- + T_{-x_R} B_+$. (So $\rho$ is obtained by splitting $1_B$ into two halves and shifting them in opposite directions by distance $R$.) 
 Then we have $W_2^2(\rho, \rho^*) = R^2$, whereas $\mathcal{E}_{W}[\rho^*] - \mathcal{E}_{W}[\rho] \ll R^2$. Therefore one has to allow the constant in \eqref{eq_w2} to go to zero as $R\to\infty$. Also note that one cannot replace the $R$ dependence by the support radius of $\rho^*$, since in this example we have $\supp\rho^* = B$.

\end{remark}

\section{Stability for the Newtonian potential}\label{sec5}

Now we turn to the special case for Newtonian potential $\mathcal{N}$, and aim to prove Theorem~\ref{theorem3}. As can be seen in \eqref{eq_newtonian}, $\mathcal{E}_\mathcal{N}$ is closely related to the $\dot H^{-1}$ norm. This allows us to use the remarkable observation by Loeper \cite{Loeper} on the connection between the 2-Wasserstein distance and $\dot H^{-1}$ norm: 

\begin{proposition}[{\cite[Proposition~2.8]{Loeper}}] \label{prop_loeper}
For $\rho_1, \rho_2$ in $\mathcal{P}_2(\mathbb{R}^n)\cap L^\infty(\mathbb{R}^n)$, we have
\[
\|\rho_1-\rho_2\|_{\dot H^{-1}(\mathbb{R}^n)} \leq \max\{\|\rho_1\|_{\infty}, \|\rho_2\|_{\infty}\}^{1/2} W_2(\rho_1, \rho_2).
\]
\end{proposition}

\medskip
Below we prove Theorem~\ref{theorem3}, which follows immediately by combining Proposition~\ref{prop_loeper} with Theorem~\ref{theorem2}. 
\begin{proof}[\textbf{\textup{Proof of Theorem~\ref{theorem3}}}]
Note that for $n\geq 1$, the Newtonian potential in $\mathbb{R}^n$ is given by $\mathcal{N} = c_n W_{-n+2}$ for some constant $c_n>0$. For any  $\rho\in L_1^+(\mathbb{R}^n)\cap L^\infty(\mathbb{R}^n)$, $\mu := \rho/\|\rho\|_1$ is a probability density, therefore we can apply 
Theorem~\ref{theorem2} to $\mu$ (and use the explicit constant $C(W_k,R)=(2R)^{k-2}$) to obtain
\begin{equation}\label{eq_thm2_temp}
 \mathcal{E}_{\mathcal{N}}[\mu^*] - \mathcal{E}_{\mathcal{N}}[\mu] \geq c(n) R^{-n}\,  W_2^2(T_{x_0}\mu, \mu^*),
\end{equation}
where we also used that $\mu$ and $\rho$ has the same support and the same center of mass $x_0$.
Applying Proposition~\ref{prop_loeper} to $\rho_1 = T_{x_0}\mu$ and $\rho_2 = \mu^*$ gives the following (note that $\rho_1, \rho_2$ both have the same $L^\infty$ norm as $\mu$):
\[
W_2^2( T_{x_0}\mu, \mu^*) \geq \|\mu\|_{\infty}^{-1} \| T_{x_0}\mu- \mu^*\|_{\dot H^{-1}(\mathbb{R}^n)}^2,
\]
thus combining it with \eqref{eq_thm2_temp} yields
\[
 \mathcal{E}_{\mathcal{N}}[\mu^*] - \mathcal{E}_{\mathcal{N}}[\mu] \geq \frac{c(n)}{\|\mu\|_{\infty} R^n}  \|T_{x_0}\mu- \mu^*\|_{\dot H^{-1}(\mathbb{R}^n)}^2 = \frac{c(n)}{\|\mu\|_{\infty} R^n}   \mathcal{E}_{\mathcal{N}}[T_{x_0}\mu - \mu^*].
\]
Finally, plugging $\mu = \rho/\|\rho\|_1$ into above gives the following inequality for $\rho$:
\[
 \mathcal{E}_{\mathcal{N}}[\rho^*] - \mathcal{E}_{\mathcal{N}}[\rho] \geq\frac{c(n)\|\rho\|_1}{\|\rho\|_{\infty} R^n} \,  \mathcal{E}_{\mathcal{N}}[T_{x_0}\rho - \rho^*],
\]
finishing the proof.
\end{proof}

The rest of this section is devoted to the proof of Theorem~\ref{theorem4}, which shows that for $n\geq 3$, the conjecture
 \eqref{conj_guo} cannot hold if the constant $c(n)$ is only allowed to depend on $n$. The density that we construct is almost the same as $1_{B(0,1)}$, plus an extra spike (with high density but a tiny mass) centered at distance 6 away from the origin.

\begin{proof}[\textbf{\textup{Proof of Theorem~\ref{theorem4}}}] For $0<\epsilon\ll 1$, let $\rho_1$, $\rho_2$ be two radially decreasing densities given by 
\[\rho_1(x) := 1_{B(0,1)}(x), \quad \rho_2(x) := \epsilon^{-(\frac{n}{2}+1)}1_{B(0,\epsilon)}(x).\]   Note that $\rho_2$ has a large $L^\infty$ norm but a small mass: namely, $\|\rho_2\|_1 = \omega_n \epsilon^{\frac{n}{2}-1}\ll 1$ since $n\geq 3$.

Let us fix $p := (6,0,\dots,0) \in \mathbb{R}^n$, and define 
\begin{equation}\label{def_rho}
\rho := \rho_1 + T_p \rho_2
\end{equation}
throughout the proof.
Since $|p|=6$, we know $\rho_1$ and $T_p \rho_2$ have disjoint supports, and $\|\rho\|_\infty = \epsilon^{-(\frac{n}{2}+1)} \gg 1$.
The goal of this proof is to show that there exist constants $C(n)>0$ and $c(n)>0$, such that the following two inequalities hold for all sufficiently small $\epsilon \in (0,1)$:
\begin{equation}\label{claim1}
\mathcal{E}_{\mathcal{N}}[\rho^*] - \mathcal{E}_{\mathcal{N}}[\rho] \leq C(n) \epsilon^{\frac{n}{2}-1}
\end{equation}
and
\begin{equation}\label{claim2}
\inf_a \mathcal{E}_\mathcal{N}[T_a\rho - \rho^*] \geq c(n).
\end{equation}
Once we obtain these two inequalities, combining them two together directly yields \eqref{eq_counterexample}.

Since both inequalities involve $\rho^*$, we start with its explicit formula. One can easily check that 
\[\rho^*(x)=\begin{cases}
\epsilon^{-(\frac{n}{2}+1)} &\text{ for } |x|\leq \epsilon,\\
1 & \text{ for } \epsilon <  |x| \leq (1+\epsilon^n)^{1/n},\\
0 & \text{otherwise}.
\end{cases}
\]
Thus we can rewrite it as
\begin{equation}\label{rho_star}
\rho^* = \epsilon^{-(\frac{n}{2}+1)} 1_{B(0,\epsilon)} + 1_{B(0,1)} + g =  \rho_2 + \rho_1 + g,
\end{equation}
where the remainder term $g$ satisfies $g=-1$ for $|x|<\epsilon$, $g=1$ for $1<|x|<(1+\epsilon^n)^{1/n}$, and $g=0$ otherwise. As a result we have $\|g\|_\infty=1$ and $\|g\|_1 = 2\omega_n \epsilon^n\ll 1$.

Now we are ready to prove \eqref{claim1}.  Let us expand its left hand side as
\begin{equation}\label{eq_diff_temp}
\begin{split}
\mathcal{E}_{\mathcal{N}}[\rho^*] - \mathcal{E}_{\mathcal{N}}[\rho] &= \mathcal{E}_{\mathcal{N}}[\rho_1 +\rho_2 + g] - \mathcal{E}_{\mathcal{N}}[\rho_1 + T_p \rho_2]\\
&= 2 \underbrace{\int (\rho_1*\mathcal{N})(\rho_2 - T_p\rho_2)dx}_{=: I_1} + \underbrace{\int (g*\mathcal{N})(2\rho_1 + 2\rho_2 + g) dx}_{=: I_2},
\end{split}
\end{equation}
where in the second equality we used that $\int \rho_2(\rho_2 * \mathcal{N})dx = \int T_p\rho_2\,(T_p\rho_2 * \mathcal{N})dx$, which follows from the fact that $\mathcal{E}_\mathcal{N}$ is invariant under translations.

We then control $I_1$ and $I_2$ as follows:
\begin{equation}\label{estimate_I1}
|I_1| \leq 2\|\rho_2\|_1 \|\rho_1*\mathcal{N}\|_\infty \leq C(n) \epsilon^{\frac{n}{2}-1},
\end{equation}
and
\[
\begin{split}
|I_2| &\leq 2\|g\|_1 \|\rho_1*\mathcal{N}\|_\infty +  \|g*\mathcal{N}\|_\infty (2\|\rho_2\|_1 + \|g\|_1)\\
& \leq C(n) \epsilon^n + C(n) \|g*\mathcal{N}\|_\infty \epsilon^{\frac{n}{2}-1}.
\end{split}
\]
Note that $\|g*\mathcal{N}\|_\infty$ satisfies the bound
\begin{equation}\label{g*N}
\|g*\mathcal{N}\|_\infty \leq \| |g|*\mathcal{N}\|_\infty \leq \int |g|^*(x)\mathcal{N}(x)dx \leq \int_{B(0,2\epsilon)} \mathcal{N}(x)dx = C(n) \epsilon^2,
\end{equation}
where the second inequality is due to the Hardy--Littlewood inequality $\int f_1 f_2 dy \leq \int f_1^* f_2^* dy$, and the third inequality is due to the fact that $|g|^*$ is bounded by 1, and has support size $2\omega_n \epsilon^n$ (hence is supported in $B(0,2\epsilon))$. Plugging this inequality into the $I_2$ estimate yields $|I_2| \leq C(n)\epsilon^{\frac{n}{2}+1}$. We then combine it with the estimate for $I_1$ in \eqref{estimate_I1} and apply these to \eqref{eq_diff_temp} to finish the proof of \eqref{claim1}.

In the rest of the proof we aim to show \eqref{claim2}. Let us take any $a\in\mathbb{R}^n$. Using the expressions for $\rho$ and $\rho^*$ in \eqref{def_rho} and \eqref{rho_star}, we have
\[
\begin{split}
T_a\rho- \rho^* &= T_a (\rho_1 + T_p \rho_2) - (\rho_1 + \rho_2 + g) \\
&= T_a \rho_1 + T_{a+p} \rho_2 - \rho_1 - \rho_2 - g.
\end{split}
\]
Thus
\begin{equation}\label{newton_diff}
\begin{split}
\mathcal{E}_{\mathcal{N}}[T_a \rho - \rho^*] &= \mathcal{E}_{\mathcal{N}}[(T_a \rho_1 - \rho_1)  + (T_{a+p} \rho_2 - \rho_2) - g]\\
&=: \underbrace{\mathcal{E}_{\mathcal{N}}[T_a \rho_1 - \rho_1]}_{=:J_1} + \underbrace{\mathcal{E}_{\mathcal{N}}[T_{a+p}\rho_2 - \rho_2]}_{=:J_2}+ J_{\text{cross}} + J_g,
\end{split}
\end{equation}
where $J_{\text{cross}}$ contains the cross terms resulted from the two parentheses in the first identity, and $J_g$ contains all terms with $g$. Let us first show that $J_{\text{cross}}$ and $J_g$ can both be made sufficiently small for $\epsilon\ll 1$.
Here the cross terms can be controlled as
\[
\begin{split}
|J_{\text{cross}}| &= \Big|\int 2(T_{a+p}\rho_2 - \rho_2) \big((T_a \rho_1 - \rho_1)*\mathcal{N}\big) dx \Big|\\
& \leq 2\|T_{a+p} \rho_2 - \rho_2\|_1 \|(T_a \rho_1 - \rho_1)*\mathcal{N}\|_\infty\\
&\leq 4 \|\rho_2\|_1 \cdot 2\|\rho_1*\mathcal{N}\|_\infty \\
&\leq C(n) \epsilon^{\frac{n}{2}-1},
\end{split}
\]
where the second-to-last inequality follows from the Hardy--Littlewood inequality that $\|T_a \rho_1 * \mathcal{N}\|_\infty \leq \|(T_a \rho_1)^* * \mathcal{N}\|_\infty = \|\rho_1*\mathcal{N}\|_\infty$.
As for the terms involving $g$, they can be written as 
\[J_g = \int (g*\mathcal{N}) (-2T_a \rho_1 + 2 \rho_1 - 2 T_{a+p} \rho_2 + 2\rho_2 + g) dx ,\] and using the bound \eqref{g*N} one directly obtains that $|J_g| \leq C(n) \epsilon^2$.

Finally we move on to the terms $J_1$ and $J_2$, which are both nonnegative since $\mathcal{E}_\mathcal{N}[f]\geq 0$ for any $f$. For any $a\in\mathbb{R}^n$, the triangle inequality gives us that $|a| + |a+p| \geq |p| = 6$, thus we either have $|a|\geq 3$, or $|a+p|\geq 3$, or both. Below we discuss these two cases respectively.

\textbf{Case 1}.  $|a|\geq 3$. In this case we will show $J_1\geq C(n) > 0$. It can be bounded below as 
\begin{equation}\label{eq1}
\begin{split}
J_1 &=  \int (T_a\rho_1*\mathcal{N})T_a \rho_1  dx + \int (\rho_1*\mathcal{N}) \rho_1  dx- 2 \int  ( \rho_1 *\mathcal{N}) T_a \rho_1 dx \\
&= 2\int  ( \rho_1 *\mathcal{N}) (\rho_1 - T_a \rho_1) dx\\
&\geq 2\|\rho_1\|_1 \left( \inf_{B(0,1)} (\rho_1 *\mathcal{N}) - \sup_{B(0,2)^c}  (\rho_1 *\mathcal{N})\right),
\end{split}
\end{equation}
where the second equality follows from $ \int (T_a\rho_1*\mathcal{N})T_a \rho_1  dx  =  \int (\rho_1*\mathcal{N}) \rho_1  dx$,  and the inequality follows from the fact that $\rho_1$ is nonnegative and supported in $B(0,1)$, whereas $\text{supp}\,T_a \rho_1 \subset B(0,2)^c$ due to $|a|>3$. 

We point out that the  right hand side of \eqref{eq1} is nonnegative since $\rho_1 * \mathcal{N}$ is radially decreasing (since it is the convolution of two radially decreasing functions). In fact, it is strictly radially decreasing: for any $r>0$, divergence theorem yields that
\[
\partial_r (\rho_1*\mathcal{N})(r) = \frac{1}{|\partial B(0,r)|} \int_{\partial B(0,r)} \nabla(\rho_1*\mathcal{N}) \cdot \vec n d\sigma = -\frac{\int_{B(0,r)} \rho_1 dx}{n \omega_n r^{n-1}} < 0.
\]
Thus $\inf_{B(0,1)} (\rho_1 *\mathcal{N}) - \sup_{B(0,2)^c}  (\rho_1 *\mathcal{N}) = c(n)>0$, and plugging it into \eqref{eq1} gives $J_1 \geq c(n)>0$.

\textbf{Case 2}. $|a+p|\geq 3$. In this case we will show $J_2\geq C(n) > 0$ for all $\epsilon\in (0,1)$ that is sufficiently small.  Using the translational invariance of $\mathcal{E}_\mathcal{N}$, we expand $J_2$ as 
\begin{equation}\label{J2}
\begin{split}
J_2 &= 2\int \rho_2 (\mathcal{N}* \rho_2)dx - 2 \int T_{a+p} \rho_2\,  (\mathcal{N}* \rho_2)dx.
\end{split}
\end{equation}
The first integral is positive and can be bounded below as
\begin{equation}\label{j21}
\int \rho_2 (\mathcal{N}* \rho_2)dx = \iint \rho_2(x)\rho_2(y) \mathcal{N}(x-y) dxdy \geq \|\rho_2\|_1^2 C(n)(2\epsilon)^{2-n} = C(n)>0, 
\end{equation}
where the second step follows from the fact that $\text{supp} \,\rho_2 = B(0,\epsilon)$, and $\mathcal{N}(x-y)\geq C(n)(2\epsilon)^{2-n}$ for all $x,y \in B(0,\epsilon)$.
As for the second integral in \eqref{J2}, it can be made sufficiently small for small $\epsilon$:
\begin{equation}\label{j22}
\int T_{a+p}\rho_2  \,(\mathcal{N}* \rho_2)dx = \iint (T_{a+p}\rho_2)(x)\rho_2(y) \mathcal{N}(x-y) dxdy \leq \|\rho_2\|_1^2 \mathcal{N}(1) \leq C(n) \epsilon^{n-2},
\end{equation}
where the first inequality follows from the fact that $\text{supp}\, \rho_2 \subset B(0,1)$ and $\text{supp}\, T_{a+p}\rho_2 \subset B(0,2)^c$ (recall that $|a+p|\geq 3$), so the two supports are disjoint with at least distance 1 from each other.
Combining \eqref{j21} and \eqref{j22} gives that $J_2 > C(n)$ for sufficiently small $\epsilon>0$.

Finally, since for any $a\in \mathbb{R}^n$, at least one of Case 1 and Case 2 must be true, we have $J_1 + J_2 \geq C(n)>0$ for  sufficiently small $\epsilon>0$ (where we also use that $J_1, J_2\geq 0$). Applying this to \eqref{newton_diff} and combining with the previous estimates $|J_{\text{cross}}| = O(\epsilon^{\frac{n}{2}-1})$ and $|J_g| = O(\epsilon^2)$ yields \eqref{claim2}, finishing the proof.
\end{proof}

\vspace{0.5cm}
\begin{tabular}{ll}
\textbf{Xukai Yan} & \textbf{Yao Yao} \\
{\small Department of Mathematics}  & {\small School of Mathematics}\\
{\small Oklahoma State University } & {\small Georgia Institute of Technology}\\
{\small 401 Mathematical Sciences Building\qquad\qquad}  & {\small 686 Cherry Street}\\
{\small Stillwater, OK 74078} & {\small  Atlanta, GA 30332}\\
{\small Email: xuyan@okstate.edu} & {\small Email: yaoyao@math.gatech.edu}\\
   & \\
\end{tabular}

\end{document}